\documentclass[10pt,a4paper] {article} 
\usepackage{latexsym,amsmath,enumerate,amssymb,amsbsy,amsthm} 
\usepackage{enumerate,verbatim,tikz} 

\newtheorem{Theorem}{Theorem}[section] 
\newtheorem{Lemma}[Theorem]{Lemma} 
 
\newtheorem{Proposition}[Theorem]{Proposition} 
\newtheorem{Corollary}[Theorem]{Corollary} \numberwithin{equation}{section} 

\theoremstyle{definition}

\newtheorem{Example}[Theorem]{Example}

\title{Vector-Circulant Matrices over Finite
Fields and Related Codes} 
\author{Somphong~Jitman\thanks{S. Jitman is with the  Department of Mathematics, Faculty of Science, 
              Silpakorn University, Nakhon Pathom 73000,  Thailand
  (email: sjitman@gmail.com).} } 
\date{}

\renewcommand\footnotemark{}
\begin{document} 
\maketitle 
\begin{abstract}

	A vector-circulant matrix is a natural generalization of the
classical  circulant matrix and  has applications in constructing
 additive codes. This article formulates the concept of a
vector-circulant matrix over finite fields and gives  an algebraic
characterization for this kind of matrix.
  Finally,  a~construction of
additive codes with vector-circulant based over $\mathbb{F}_4$ is
given together with some examples of good half-rate additive
codes.
\end{abstract}

\section{Introduction}

  Let $q$ be a   power of a prime number and  $n$ be a
positive integer. Denote by $\mathbb{F}_q$ the finite field of
order $q $ and $M_n(\mathbb{F}_q)$ the $\mathbb{F}_q$-algebra  of
all $n\times n$ matrices whose entries are in~$\mathbb{F}_q$.
Given
  $\alpha\in \mathbb{F}_q\setminus \{0\}$, a matrix $A\in M_n(\mathbb{F}_q)$
is said to be $\alpha$-\textit{twistulant} [3] if \bigskip
\[ A=\left[%
\begin{array}{ccccc}
  a_0 & a_1 & \dots & a_{k-2}& a_{n-1} \\
   \alpha a_{n-1}  &  a_0  & \dots & a_{n-3}&   a_{n-2}   \\
   \alpha a_{n-2}  &  \alpha  a_{n-1}  & \dots & a_{n-4} &   a_{n-3}   \\
  \vdots & \vdots & \ddots & \vdots & \vdots \\
    \alpha a_1    &   \alpha a_2  & \dots &  \alpha a_{n-1}&   a_0  \\
\end{array}%
\right].\] Such a matrix is called \textit{circulant} [resp.,
\textit{negacirculant}] \textit{matrix} when $\alpha=1$ [resp.,
$\alpha=-1$]. The set of all $n\times n$ circulant [resp.,
$\alpha$-twistulant, \mbox{negacirculant}] matrices over
$\mathbb{F}_q$ is isomorphic to $\mathbb{F}_q[x]/\langle
x^n-1\rangle$ [resp., \mbox{$\mathbb{F}_q[x]/\langle
x^n-\alpha\rangle$,} $\mathbb{F}_q[x]/\langle x^n+1\rangle$] as
commutative algebras [3].

Circulant matrices over finite fields and their   well-known
generalizations in the notion of twistulant and negacirculant
\mbox{matrices} have
  widely been applied in many branches of Mathematics.
  Recently, they have been applied to construct
circulant based additive codes [5] and double circulant codes [4]
with optimal and extremal parameters.

In Section 2, we generalize the concept of circulant matrix over
the finite field~$\mathbb{F}_q$ and call it a vector-circulant
matrix. The algebraic structure of  the set of these matrices is
investigated. This generalization leads to a construction of
vector-circulant based additive codes over $\mathbb{F}_4$ in
Section 3. Examples are some optimal codes are also demonstrated
here.

\section{Vector-Circulant Matrices over Finite Fields}

\noindent Given a vector
$\boldsymbol{\lambda}=(\lambda_0,\lambda_1,\dots,\lambda_{n-1})\in
\mathbb{F}_q^n$,   let $\rho_{\boldsymbol{\lambda}}:
\mathbb{F}_q^n \rightarrow \mathbb{F}_q^n  $ be   defined by
\begin{eqnarray}\label{rho}
&\rho_{\boldsymbol{\lambda}}((v_0,v_1,\dots, v_{n-1})) &= (0,
v_0,v_1,\dots,
v_{n-2})+v_{n-1} \boldsymbol{\lambda}\notag\\
&&=(v_{n-1}\lambda_0,v_0+v_{n-1}\lambda_1,\dots,v_{n-2}+v_{n-1}\lambda_{n-1}).
\end{eqnarray}
The map $\rho_{\boldsymbol{\lambda}}$ is called the
$\boldsymbol{\lambda}$-\textit{vector-cyclic shift}
 on
$\mathbb{F}_q^n $.

A  matrix $A\in M_n(\mathbb{F}_q)$ is said to be
\textit{vector-circulant}, or specifically,
$\boldsymbol{\lambda}$-\textit{vector-circulant} if
\[A=\left[%
\begin{array}{r }
  a_0~~a_1 ~~ \cdots ~~ a_{n-1}~ \\
\rho_{\boldsymbol{\lambda}}(a_0~~a_1 ~~ \cdots ~~ a_{n-1} )\\
 \rho_{\boldsymbol{\lambda}} ^2(a_0~~a_1 ~~ \cdots ~~ a_{n-1} )\\
\vdots~~~~~~~~~~~~~\\
 \rho_{\boldsymbol{\lambda}} ^{n-1}(a_0~~a_1 ~~ \cdots ~~ a_{n-1} )\\
\end{array}%
\right]=:{\rm
cir}_{\boldsymbol{\lambda}}(a_0,a_1,\dots,a_{n-1}).\]  Clearly, a
 $ {\boldsymbol{\lambda}}$-{vector-circulant} matrix becomes  the
classical circulant    [resp.,  $\alpha$-twistualnt] matrix when $
\boldsymbol{\lambda} $ is the vector $(1,0,\dots,0)$ [resp,
$(\alpha,0,\dots,0)$ ].
\begin{Example}
Consider the finite field
$\mathbb{F}_4=\{0,1,\alpha,\alpha^2=1+\alpha\}$.   The matrices
\begin{eqnarray*}
 \left[%
\begin{array}{ccc}
  1 & \alpha & 0  \\
  0 & 1 & \alpha   \\
  \alpha & 0 & \alpha^2  \\
\end{array}%
\right]={\rm cir}_{(1,0,1)}(1,\alpha,0 ) \end{eqnarray*}  and
\begin{eqnarray*}
  \left[%
\begin{array}{cccc}
  1 & \alpha & 0 & \alpha \\
  \alpha^2  & 1 & \alpha  & \alpha \\
  \alpha^2 & \alpha & 1 & 0  \\
  0 & \alpha^2  & \alpha  & 1 \\
\end{array}%
\right]={\rm cir}_{(\alpha,0,0,1)}(1,\alpha,0,\alpha)
\end{eqnarray*} are  $3\times 3$ and $4\times 4$ vector-circulant matrices, respectively. They are obviously not circulant.
\end{Example}

From (\ref{rho}), it is easily verified that
$\rho_{\boldsymbol{\lambda}}$ is an $\mathbb{F}_q$-linear
transformation corresponding to
\[T_{\boldsymbol{\lambda}}= \left[%
\begin{array}{ccccc}
  0 & 1 & 0 & \dots & 0 \\
  0 & 0 & 1 & \dots & 0 \\
  \vdots& \vdots & \vdots & \ddots & \vdots \\
  0 & 0 & 0 & \dots & 1 \\
  \lambda_0 & \lambda_1 & \lambda_2 & \dots & \lambda_{n-1} \\
\end{array}%
\right],\] i.e., $\rho_{\boldsymbol{\lambda}}
({\boldsymbol{v}})={\boldsymbol{v}}T_{\boldsymbol{\lambda}}$, for
all $\boldsymbol{v}\in \mathbb{F}_q^n$. Consequently, for $1\leq
i$, $\rho_{\boldsymbol{\lambda}}^i$ is an $\mathbb{F}_q$-linear
transformation corresponding to $T_{\boldsymbol{\lambda}}^i$. Hence,
\begin{eqnarray}\label{sum}
  {\rm
cir}_{\boldsymbol{\lambda}}(a_0,a_1,\dots,a_{n-1})=
\displaystyle\sum_{i=0}^{n-1} a_i{\rm
cir}_{\boldsymbol{\lambda}}(E_{i+1}) ,
\end{eqnarray}
where   $E_i=(0,\dots,0,\underbrace{1}_{i^{th}},0,\dots,0)$, for
$1\leq i\leq n$.

Observe that $T_{\boldsymbol{\lambda}}$ need not be invertible.
For $\boldsymbol{\lambda}=(\lambda_0,
\lambda_1,\dots,\lambda_{n-1})$, the singularity of
$T_{\boldsymbol{\lambda}}$ depends on ${{\lambda}_0}$. By applying
a suitable sequence of elementary row operations,
$T_{\boldsymbol{\lambda}}$ is equivalent to  an $n\times n$
diagonal matrix ${\rm diag}(\lambda_0,1,1,\dots,1)$. Then the next
proposition follows.
\begin{Proposition}  $T_{(\lambda_0,\lambda_1,\dots,\lambda_{n-1})}$ is
invertible if and only if $\lambda_0\neq 0$.
\end{Proposition}
The set of all $n\times n$ $
{\boldsymbol{\lambda}}$-vector-circulant matrices
over~$\mathbb{F}_q$ is denoted by ${\rm
Cir}_{n,{\boldsymbol{\lambda}}}(\mathbb{F}_q)$. Consider
$M_n(\mathbb{F}_q)$ as an algebra over $\mathbb{F}_q$, our goal is
to show that $ {\rm Cir}_{n,{\boldsymbol{\lambda}}}(\mathbb{F}_q)$
is an abelian subalgebra of $M_n(\mathbb{F}_q)$. It follows
directly from the linearity of $\rho_{\boldsymbol{\lambda}}$ that
$ {\rm Cir}_{n,{\boldsymbol{\lambda}}}(\mathbb{F}_q)$ is a
subspace of the  $\mathbb{F}_q$-vector space  $M_n(\mathbb{F}_q)$.
Moreover, by application of (\ref{rho}), the set $\{ {\rm
cir}_{\boldsymbol{\lambda}}(E_1), {\rm
cir}_{\boldsymbol{\lambda}}(E_2), \dots, {\rm
cir}_{\boldsymbol{\lambda}}(E_{n})\}$ can be verified to be a
basis of $ {\rm Cir}_{n,{\boldsymbol{\lambda}}}(\mathbb{F}_q)$. To
prove that $ {\rm Cir}_{n,{\boldsymbol{\lambda}}}(\mathbb{F}_q)$
is a commutative ring, we need the following lemma and corollary.
For convenience, let $T_{\boldsymbol{\lambda}}^0$ denote the
identity matrix $I_n$.

\begin{Lemma}\label{lem:2.4}
\begin{enumerate}[$i)$]
\item $ T_{\boldsymbol{\lambda}}^{m}={\rm
 cir}_{\boldsymbol{\lambda}}(\rho_{\boldsymbol{\lambda}}^m(E_{1}))$, for all integers $  0\leq m $.
\item $T_{\boldsymbol{\lambda}}^{i} = {\rm cir}_{
\boldsymbol{\lambda}}(E_{i+1})$, for all   $0\leq i< n$.
\end{enumerate}
\end{Lemma}
\begin{proof}
First, we prove $i)$  by induction on $m$. By the definition,
$T_{\boldsymbol{\lambda}}^0=I_n={\rm
 cir}_{\boldsymbol{\lambda}}(\rho_{\boldsymbol{\lambda}}^0(E_{1}))$.
 Clearly, $T_{\boldsymbol{\lambda}}  ={\rm
 cir}_{\boldsymbol{\lambda}}(\rho_{\boldsymbol{\lambda}}
 (E_{1}))$. Assume that $T_{\boldsymbol{\lambda}} ^k ={\rm
 cir}_{\boldsymbol{\lambda}}(\rho_{\boldsymbol{\lambda}}^k
 (E_{1}))$ for all positive integers $k<m$.
 Then
 \begin{eqnarray*}
&T_{\boldsymbol{\lambda}} ^{k+1} &= {\rm
cir}_{\boldsymbol{\lambda}}(\rho_{\boldsymbol{\lambda}}^k
 (E_{1}))T_{\boldsymbol{\lambda}}\\
 &&=\left[%
\begin{array}{c}
 \rho_{\boldsymbol{\lambda}}^k
 (E_{1})   \\
 \rho_{\boldsymbol{\lambda}}^{k+1}
 (E_{1})   \\
  \vdots  \\
  \rho_{\boldsymbol{\lambda}}^{k+n-1}
 (E_{1})  \\
\end{array}%
\right]T_{\boldsymbol{\lambda}}\\
&&=\left[%
\begin{array}{c}
 \rho_{\boldsymbol{\lambda}}^{k+1}
 (E_{1})   \\
 \rho_{\boldsymbol{\lambda}}^{k+2}
 (E_{1})   \\
  \vdots  \\
  \rho_{\boldsymbol{\lambda}}^{k+n}
 (E_{1})  \\
\end{array}
\right] \\
&&={\rm
cir}_{\boldsymbol{\lambda}}(\rho_{\boldsymbol{\lambda}}^{k+1}
 (E_{1})).
 \end{eqnarray*}
 Hence $i)$ is proved.

 Note that,  for all $1\leq i\leq
  n$, $\rho_{ \boldsymbol{\lambda}}^{i-1}(E_1)=E_i$. Hence,     $ii)$ follows immediately from $i)$.
\end{proof}
The next corollary is a direct consequence of Lemma \ref{lem:2.4}.
\begin{Corollary}\label{cor:2.6}
  $T_{\boldsymbol{\lambda}}^m\in {\rm
Cir}_{n,\boldsymbol{\lambda}}(\mathbb{F}_q)$, for all   $0\leq m$.
\end{Corollary}


\begin{Theorem}\label{thm:2.8}  $ {\rm Cir}_{n,{\boldsymbol{\lambda}}}(\mathbb{F}_q) $ is
a  commutative  subring of $M_n(\mathbb{F}_q)$.
\end{Theorem}
\begin{proof} Since $ ({\rm Cir}_{n,{\boldsymbol{\lambda}}}(\mathbb{F}_q),+) $ is an additive subgroup of $M_n(\mathbb{F}_q)$ containing $I_n= T^0_{\boldsymbol{\lambda}}$, it is sufficient to show that   $ {\rm
Cir}_{n,\boldsymbol{\lambda}}(\mathbb{F}_q) $ is closed under the
usual multiplication of matrices.  Let $(a_0,a_1,\dots, a_{n-1}),$
$(b_0,b_1,\dots, b_{n-1})\in \mathbb{F}_q^n$. Then
\begin{eqnarray}\label{eq:mul}
&{\rm cir}_{\boldsymbol{\lambda}}(a_0,a_1,\dots, a_{n-1}){\rm
cir}_{\boldsymbol{\lambda}}(b_0,b_1,\dots, b_{n-1}
&=\left(\displaystyle\sum_{i=0}^{n-1}a_i{\rm cir}_{\boldsymbol{\lambda}}(E_{i+1}) \right)\left(\displaystyle\sum_{j=0}^{n-1}b_j{\rm cir}_{\boldsymbol{\lambda}}(E_{j+1}) \right), \notag\\
&&~~~~~~~~~~~~~~~~~~~~~~~~~~~~~ \text{by (\ref{sum}),}\notag\\
&&=\displaystyle\sum_{i=0}^{n-1}\displaystyle\sum_{j=0}^{n-1}a_ib_j
{{\rm
 cir}_{\boldsymbol{\lambda}}(E_{i+1})\rm
 cir}_{\boldsymbol{\lambda}}(E_{j+1})\notag\\
 &&=\displaystyle\displaystyle\sum_{i=0}^{n-1}\displaystyle\displaystyle\sum_{j=0}^{n-1}a_ib_j T^iT^j, ~~~\text{by Lemma \ref{lem:2.4},} \notag\\
 &&=\displaystyle\displaystyle\sum_{i=0}^{n-1}\displaystyle\displaystyle\sum_{j=0}^{n-1}a_ib_j T^{i+j}\\
 &&\in {\rm Cir}_{n,{\boldsymbol{\lambda}}}(\mathbb{F}_q) ,
 ~~~~~~~~~~~~\text{by Corollary \ref{cor:2.6}}.\notag
\end{eqnarray}
From (\ref{eq:mul}), the commutativity is obvious.
\end{proof}


\begin{Corollary}
$ {\rm Cir}_{n,\boldsymbol{\lambda}}(\mathbb{F}_q) $   is a
commutative algebra.
\end{Corollary}
\begin{proof}
  It follows from the fact that $ {\rm Cir}_{n,\boldsymbol{\lambda}}(\mathbb{F}_q)
  $ is an $\mathbb{F}_q$-vector space and Theorem \ref{thm:2.8}.
\end{proof}

 For
$\boldsymbol{\lambda}=(\lambda_0,\lambda_1,\dots,\lambda_{n-1})$,
let $ \boldsymbol{\lambda}(x)= \lambda_0+\lambda_1
x+\dots+\lambda_{n-1} x^{n-1}$ be the corresponding polynomial
representation of  $ \boldsymbol{\lambda}$.

The next theorem is an algebraic characterization of $ {\rm
Cir}_{n,\boldsymbol{\lambda}}(\mathbb{F}_q) $.

\begin{Theorem}
 $ {\rm Cir}_{n,\boldsymbol{\lambda}}(\mathbb{F}_q) $   is isomorphic to $\mathbb{F}_q[x]/\langle
x^n-\boldsymbol{\lambda}(x)\rangle$ as  commutative algebras.
\end{Theorem}
\begin{proof}
Defined $\varphi :{\rm
Cir}_{n,\boldsymbol{\lambda}}(\mathbb{F}_q)\rightarrow
\mathbb{F}_q[x]/\langle x^n-\boldsymbol{\lambda}(x)\rangle $ by
\[ {\rm cir}_{ \boldsymbol{\lambda}}(a_0,a_1,\dots, a_{n-1})
\mapsto \displaystyle\sum_{i=0}^{n-1} a_ix^i+\langle
x^n-\boldsymbol{\lambda}(x)\rangle.\]It is easily seen that
$\varphi$ is an additive group isomorphism. Let $(a_0,a_1,\dots,
a_{n-1}),$ $(b_0,b_1,\dots, b_{n-1})\in \mathbb{F}_q^n$ and
$\alpha\in \mathbb{F}_q$. Then
\begin{eqnarray}\label{eq:scalarmul}
 & \varphi(\alpha {\rm cir}_{ \boldsymbol{\lambda}} (a_0,a_1,\dots, a_{n-1}))
  &=\varphi( {\rm cir}_{ \boldsymbol{\lambda}}(\alpha a_0,\alpha a_1,\dots, \alpha  a_{n-1}))\notag\\
  &&=\displaystyle\sum_{i=0}^{n-1}\alpha a_ix^i +\langle
x^n-\boldsymbol{\lambda}(x)\rangle\notag\\
  &&=\alpha\left(\displaystyle\sum_{i=0}^{n-1} a_ix^i +\langle
x^n-\boldsymbol{\lambda}(x)\rangle\right)\notag\\
 &&=\alpha  \varphi( {\rm cir}_{ \boldsymbol{\lambda}} (a_0,a_1,\dots,
 a_{n-1})).
\end{eqnarray}
By (\ref{eq:mul}) and Lemma \ref{lem:2.4}, we then have
\begin{eqnarray*}
&&\varphi({\rm cir}_{ \boldsymbol{\lambda}}( a_0,a_1,\dots,
a_{n-1} )  {\rm cir}_{ \boldsymbol{\lambda}}( b_0,b_1,\dots,
b_{n-1})
)\\&&~~~~~~~~~~~~~~=\varphi(\displaystyle\sum_{i=0}^{n-1}\displaystyle\sum_{j=0}^{n-1}a_ib_j
T^{i+j}),~~~~~~~~~~~~~~~~~~~~\text{by
 (\ref{eq:mul})},\\
 &&~~~~~~~~~~~~~~=\varphi(\displaystyle\sum_{i=0}^{n-1}\displaystyle\sum_{j=0}^{n-1}a_ib_j{\rm
 cir}_{\boldsymbol{\lambda}}(\rho_{\boldsymbol{\lambda}}^{i+j}(E_{1}))), ~~~~~~~\text{by Lemma \ref{lem:2.4},}\\
 &&~~~~~~~~~~~~~~= \displaystyle\sum_{i=0}^{n-1}\displaystyle\sum_{j=0}^{n-1}a_ib_j\varphi({\rm
 cir}_{\boldsymbol{\lambda}}(\rho_{\boldsymbol{\lambda}}^{i+j}(E_{1}))),~~~~~~~\text{by
 (\ref{eq:scalarmul})}, \\
 &&~~~~~~~~~~~~~~= \displaystyle\sum_{i=0}^{n-1}\displaystyle\sum_{j=0}^{n-1}a_ib_j x^{i+j}+\langle x^n-\boldsymbol{\lambda}(x)\rangle\\
&&~~~~~~~~~~~~~~=\left(\displaystyle\sum_{i=0}^{n-1} a_ix^i
+\langle
x^n-\boldsymbol{\lambda}(x)\rangle\right)\left(\displaystyle\sum_{j=0}^{n-1}
a_jx^j +\langle x^n-\boldsymbol{\lambda}(x)\rangle\right)\\
 &&~~~~~~~~~~~~~~= \varphi({\rm
cir}_{ \boldsymbol{\lambda}}( a_0,a_1,\dots, a_{n-1} ) )
\varphi({\rm cir}_{ \boldsymbol{\lambda}}( b_0,b_1,\dots, b_{n-1})
).
\end{eqnarray*}
This completes the proof.
\end{proof}

\section{Vector-Circulant Based Additive Codes
over~$\mathbb{F}_4$}

In this section, we restrict our study to the finite field of $4$
elements $\mathbb{F}_4 =\{0,1,\alpha,\alpha^2=1+\alpha\}$.
 A code  of length $n$ over
$\mathbb{F}_4$ is defined to be a non-empty subset of
$\mathbb{F}_4^n$. A code $C$ is said to be additive if it is an
additive subgroup of the additive group $(\mathbb{F}_4^n,+)$.
Throughout, every code is assume to be additive. It is know [5]
that $C$ contains $2^k$ codewords for some $0 \leq k \leq 2n$, and
can be defined by a $k \times n$ generator matrix, with entries
from $\mathbb{F}_4$, whose rows span $C$ additively. We regard an
additive code of length $n$ over $\mathbb{F}_4$ containing $2^k$
codewords as an $(n, 2^k)$ code.
 The Hamming weight of $v \in
\mathbb{F}_4^n$, denoted ${\rm wt}(\boldsymbol{v})$, is defined to
be the number of nonzero components of $v$. The Hamming distance
between $u\neq v\in \mathbb{F}_4^n$ is defined as ${\rm wt}(u -
v)$. The minimum distance of the code $C$, denoted by $d(C)$, is
the minimal Hamming distance between any two distinct codewords of
$C$. As  $C$ is additive, the minimum distance equals the smallest
nonzero weight of any codewords in $C$. An $(n,2^k)$ code with
minimum distance $d$ is called an $(n, 2^k, d)$ code.

We focus on a construction of additive codes for the particular
case $k=n$, or half-rate codes, or equivalently,   $(n,2^n)$
codes. It follows from the Singleton bound [2] that any half-rate
additive code over $\mathbb{F}_4$ must satisfy
\[ d\leq \left\lfloor
\frac{n}{2}\right\rfloor +1.\] An $(n,2^n)$  code $C$ is said to
be  \textit{extremal} if it attains the equality in the Singleton
bound, and \textit{near-extremal} if it has minimum distance
$\left\lfloor \frac{n}{2}\right\rfloor$.

 Given $\boldsymbol{\lambda}\in \mathbb{F}_4^n$, a $\boldsymbol{\lambda}$-\textit{vector-circulant based additive code} is defined to be an additive code
generated by a $\boldsymbol{\lambda}$-circulant generator matrix
of the following form:
\begin{eqnarray*}
\left[%
\begin{array}{r }
  a_0~~a_1 ~~ \cdots ~~ a_{n-1}~ \\
\rho_{\boldsymbol{\lambda}}(a_0~~a_1 ~~ \cdots ~~ a_{n-1} )\\
 \rho_{\boldsymbol{\lambda}} ^2(a_0~~a_1 ~~ \cdots ~~ a_{n-1} )\\
\vdots~~~~~~~~~~~~~\\
 \rho_{\boldsymbol{\lambda}} ^{n-1}(a_0~~a_1 ~~ \cdots ~~ a_{n-1} )\\
\end{array}%
\right]=:{\rm cir}_{\boldsymbol{\lambda}}(a_0,a_1,\dots,a_{n-1})
\end{eqnarray*}
Such a code  is called a circulant based additive code if
$\boldsymbol{\lambda}=(1,0,\dots,0)$ and it is called a
$\alpha$-twistulant based additive code if
$\boldsymbol{\lambda}=(\alpha,0,\dots,0)$.

An advantage of this construction is that there are typically much
more   additive codes than circulant based or twistulant based
additive codes [5].

We implement  a   procedure in the computer algebra system Magma
[1] to construct   vector-circulant based  additive codes. Based
on this construction, we search for     half-rate additive codes
with highest minimum distances of length up to $13$. The result is
shown in Table~\ref{tb1};  codes of length $2$ to $7$ are extremal
and codes of length $8$ to $13$ are near-extremal.

\begin{table}[H]
  \centering  \caption{Vector-circulant based  $(n,2^n)$   codes $C$ over  $\mathbb{F}_4 =\{0,1,\alpha,\alpha^2=1+\alpha\}$   { generated by }  ${\rm cir} _{\boldsymbol{\lambda}}(\boldsymbol{v}) $ with highest minimum distances}\label{tb1}
  \begin{tabular}{ lllc }
    \hline
    \hline
    $n$  & $\boldsymbol{\lambda}$ & $\boldsymbol{v}$ &  $d(C)$  \\
    \hline
    $2$&$(  1  , 1)$                                                            &$(  \alpha ,  1)$& $2$\\
    $3$&$(  1  , 0 ,  \alpha) $                                                 &$(  \alpha ,  1 ,  1) $& $2$\\
    $4$&$(  1,   0,   0,   1) $                                                 &$(  1,   \alpha,   1,   1)$& $3$\\
    $5$&$(  1,   0,   0,   0,   \alpha) $                                       &$(  1,   0,   \alpha,   1,   1) $& $3$\\
    $6$&$(  1,   0,   0,   0,   0,   0) $                                       &$(  \alpha, \alpha^2,   \alpha,   1,   1,   1) $& $4$\\
    $7$&$(  1,   0,   1,   0,   0,   0,   0) $                                  &$(  0,   1,   \alpha,   1,   1,   1,   1) $& $4$\\
    $8$&$(  1,   0,   0,   0,   0,   0,   0,   \alpha) $                        &$(  0,   \alpha, \alpha^2, \alpha^2,   1,   1,   1,   1) $& $4$\\
    $9$&$(  1,   0,   0,   0,   0,   0,   0,   0,   1) $                        &$(a^2,   \alpha,   1,   1,   1,   1,   1,   1,   1)$& $4$\\
    $10$&$(  1,   0,   0,   0,   0,   0,   0,   0,   0,   0) $                  &$(  0,   \alpha,   \alpha,   1,   \alpha,   1,   1,   1,   1,   1) $& $5$\\
    $11$&$(  1,   0,   0,   0,   1,   0,   0,   0,   0,   0,   0) $             &$(  0,   \alpha, \alpha^2,   \alpha,   1,   1,   1,   1,   1,   1,   1) $& $5$\\
    $12$&$(  1,   0,   0,   0,   0,   0,   0,   0,   0,   0,   0,   0) $        &$(  0,   1, \alpha^2, \alpha^2,   1,   \alpha,   1,   1,   1,   1,   1,   1) $& $6$\\
    $13$&$(  1,   0,   0,   0,   0,   0,   0,   0,   0,   0,   0,   0,   0) $   &$(  0,   \alpha, \alpha^2,   1,   1,   \alpha,   1,   1,   1,   1,   1,   1 ,  1) $& $6$\\

    \hline
  \end{tabular}
 \end{table}



\begin{thebibliography}{800}


 
\bibitem{1} Bosma W, Cannon J J,   Playoust C. The Magma algebra system I:
the user language. \textit{Journal of Symbolic Computation},
\textbf{24}: 235--265 (1997)


\bibitem{2} Danielsen L E,  Parker M G. Directed graph representation of
half--rate additive codes over ${\rm GF}(4)$. {\texttt http://
arxiv.org/abs/0902.3883v3} (preprint)

\bibitem{3} Davis P J.  \textit{Circulant Matrices}. Chelesa publishing, New
York, second edition, 1994

\bibitem{4} Grassl M,  Gulliver A. On circulant self-dual codes over
small fields. \textit{Designs, Codes and Cryptography},
\textbf{52}: 57--81  (2009)

\bibitem{5} Gulliver T A, Kim J-L.  Circulant based extremal
additive self-dual codes over $GF(4)$. \textit{IEEE Transaction on
Information Theory}, \textbf{50}: 359--366 (2004)

	
 
	
\end{thebibliography}
\end{document}